\documentclass[reqno]{amsart}
\usepackage[utf8]{inputenc}
\usepackage{amsmath,amssymb,amsthm}
\usepackage{tikz}
\usetikzlibrary{arrows}
\usepackage{url}
\usepackage[utf8]{inputenc}
\usepackage{fullpage}

\newtheorem{theorem}{Theorem}[section]

\newtheorem{question}[theorem]{Question}

\newtheorem{corollary}[theorem]{Corollary}

\newcommand{\Z}{\mathbb{Z}}

\renewcommand{\P}{\mathbb{P}}
\newcommand{\N}{\mathbb{N}}
\newcommand{\li}{\text{li}}
\newcommand\numberthis{\addtocounter{equation}{1}\tag{\theequation}}

\title{Random Multiplicative Walks on the Residues Modulo $\mathbf{n}$}
\author{Nathan McNew}

\begin{document}
\begin{abstract}
    We introduce a new arithmetic function $a(n)$ defined to be the number of random multiplications by residues modulo $n$ before the running product is congruent to 0 modulo $n$.  We give several formulas for computing the values of this function and analyze its asymptotic behavior.  We find that it is closely related to $P_1(n)$, the largest prime divisor of $n$.  In particular, $a(n)$ and $P_1(n)$ have the same average order asymptotically.  Furthermore, the difference between the functions $a(n)$ and $P_1(n)$ is $o(1)$ as $n$ tends to infinity on a set with density approximately $0.623$.  On the other hand however, we see that (except on a set of density zero) the difference between $a(n)$ and $P_1(n)$ tends to infinity on the integers outside this set.  Finally we consider the asymptotic behaviour of the difference between these two functions and find that $\sum_{n\leq x}\big( a(n)-P_1(n)\big) \sim \left(1-\frac{\pi}{4}\right)\sum_{n\leq x} P_2(n)$, where $P_2(n)$ is the second largest divisor of $n$.
\end{abstract}
\maketitle

\section{Background} \label{background}
Consider a ``multiplicative'' random walk on the set $\Z/n\Z$ of residues modulo $n$.  In each step of this walk a residue modulo $n$ is chosen, uniformly at random, and the current state is multiplied by it.  Because $\Z/n\Z$ does not form a group, but rather a monoid  under multiplication, this is not a transitive walk.  In fact, it is an absorbing random walk with a single absorbing state, $0 \pmod{n}$. So if the walk proceeds long enough then, with probability one, the walk will eventually arrive (and stay) at the residue 0.

For example, when $n=6$ this random walk is equivalent to a random walk on the directed graph depicted in Figure \ref{fig:graphneq6}. In this walk each of the edges leaving a vertex are equally likely to be chosen.  Note that it is possible to step from the residues 1 and 5 (the units modulo 6) to any of the residues modulo 6, but this is not true for the other residues, and that the only edge leaving 0 is a loop.

\begin{figure}[h!]
    \centering
    \vspace{-4mm}
    \begin{tikzpicture}[->,>=stealth',auto,
  thick,main node/.style={circle,draw}]
      \node[main node] (a1) at (-3,-1) {1};
      \node[main node] (a5) at (-3,1) {5};
      \node[main node] (a2) at (-1,0) {2};
      \node[main node] (a4) at (-1,2.0){4};
      \node[main node] (a3) at (1,-1.25) {3};
      \node[main node] (a6) at (3,0) {0};
      
{\path[every node/.style={font=\sffamily\scriptsize}]
    (a1) edge [in=270,out=210,loop,looseness=5] node[below]{} (a1)
        edge [out=20,bend left=20] node[above,pos=0.7]{} (a2)
        edge [bend right=22] node[above] {} (a3)
        edge [bend left=8] node[above,pos=0.7] {} (a4)
        edge [bend right=20] node[left,pos=0.7] {} (a5)
        edge [bend right=62,looseness=1] node[right,pos=0.85] {} (a6);}
{\path[every node/.style={font=\sffamily\scriptsize}]
    (a2) edge [in=270,out=210,loop,looseness=5] node [below,pos=0.5] {} (a2)
        edge [bend right=20] node[right,pos=0.7] {} (a4)
        edge [bend right=20] node[above,pos=0.7] {} (a6);} 
{\path[every node/.style={font=\sffamily\scriptsize}]
    (a3) edge [in=280,out=220,loop,looseness=5] node [below,pos=0.35] {} (a3)
        edge [bend right=11] node[above,pos=0.30] {} (a6);}
{\path[every node/.style={font=\sffamily\scriptsize}]
    (a5) edge [in=220,out=160,loop,looseness=5] node [below,pos=0.5]{} 
    (a5)
        edge [bend right=20] node[right] {} (a1)
        edge [bend left=70,looseness=1.2] node[above,pos=0.7] {} (a6)
        edge [bend left=15] node[right] {} (a2)
        edge [bend right=29] node[right] {} (a3)
        edge [bend left=10] node[right] {} (a4)
    (a4) edge [in=20,out=80,loop,looseness=5] node [below,pos=0.5]{}
    (a4)
        edge [bend right=20] node[right,pos=0.7] {} (a2)
        edge [bend left=25] node[above,pos=0.7] {} (a6)
    
    (a6) edge [in=0,out=60,loop,looseness=5] node [below,pos=0.5]{} (a6);}        
    \end{tikzpicture} 
    \vspace{-7mm}
    \caption{The directed graph depicting the random walk on the residues modulo 6.}
    \label{fig:graphneq6}
\end{figure}
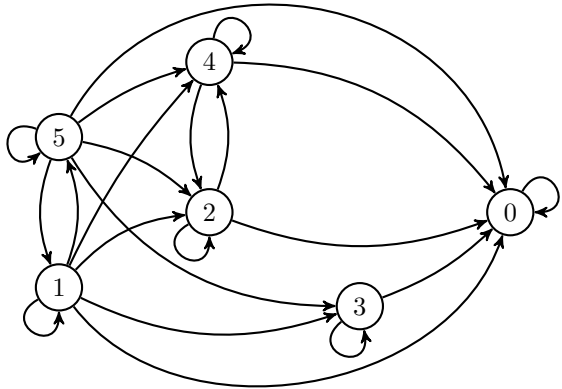

In this paper we are interested in the expected time to absorption, the number of steps it takes on average to reach $0 \pmod{n}$, if the initial state of the walk is $1 \pmod{n}$, or equivalently any unit modulo $n$.  For any positive integer $n$ we denote by $a(n)$ the expected time to absorption in the monoid $\Z/n\Z$, and study the behaviour of this arithmetic function.  Analyzing the walk depicted for $n=6$ above, one finds, for example, that $a(6)= 3.5$.

The problem of studying random walks on groups has been studied extensively, see for example \cite{Aldous}, \cite{Hildebrand}.  Many common random walks can in fact be viewed as walks on groups, for example the walk on the infinite line ($\Z$) or on a cycle ($\Z/n\Z$ under addition).  More recently, several authors have considered random walks on monoids \cite{Gretete}, \cite{Mairesse}.  Note that in the literature a random walk on a group or monoid is comprised of steps which are taken from some generating set of the group (or monoid).  In our case we take the generating set to be the entire monoid $\Z/n\Z$.  In the case of a random walk on a group, one would not generally take the entire group as a generating set as it would then be possible to transition between any two states in a single step.

We could alternatively view our problem as a random walk on the positive integers, $\N$, in which we repeatedly multiply the current state by an integer chosen uniformly at random from the range $[1,n]$ (or $[1,mn]$ where $m\geq 1$ is an integer). Our function $a(n)$ is then the expected number of multiplications before the product is divisible by $n$. 

The behavior of this function turns out to depend sensitively on the largest prime divisors of $n$. If  $n = p_1^{\alpha_1}p_2^{\alpha_2}\cdots p_k^{\alpha_k}$ with $p_1>p_2 > \ldots > p_k$, we denote by \[P_i(n) = \begin{cases} p_i & 1 \leq i \leq k \\ 1 & i>k \end{cases} \] and by $B(n)$ the sum of the prime divisors of $n$ taken with multiplicity, \[B(n) = \sum_{i=1}^k \alpha_i p_i.\]

As we will see in section \ref{sec:upplowbds} our function $a(n)$ is bounded between the functions $P_1(n)$ and $ B(n)$.  The behaviour of these functions (along with a handful of similarly defined functions) has been studied extensively (see \cite{JmdkIvicAverage,KnuthPardo,AlladiErdos,AlladiErdos2,NaslundMean,Kemeny,ErdosIvic,ErdosIvicPom}), much of it emphasizing how remarkably similar these two functions are.  Given that our function is closely related to these two functions, we analyze it similarly.

First, however, in Sections \ref{sec:recursiveeqn} and \ref{sec:squarefree} we obtain several expressions for $a(n)$.  Theorem \ref{thm:recurform} shows how to calculate $a(n)$ for any integer $n$, namely \begin{align*}
a(n) &= \frac{1}{n-\varphi(n)}\left(n+ \sum_{\substack{d|n\\d\neq n}}\varphi\left(d\right)a\left(d\right)\right).
\end{align*}
Unfortunately, this expression is recursive, but we are able to obtain expressions that aren't recursive in special cases. For prime powers (Theorem \ref{thm:primepow}) we show that $a(p^k) = k(p-1)+1$, and for squarefree values of $n$ we show (Theorem \ref{thm:squarefree}) that \[a(n) = \sum_{\substack{d|n\\d\neq 1}} (-1)^{\omega(d)+1}\frac{d}{d-\varphi(d)}.\]

In Section \ref{sec:upplowbds} we begin our study of the asymptotic behavior of this function by obtaining the average order of $a(n)$, (Corollary \ref{cor:avgord}) \[\sum_{n<x} a(n) \sim \frac{\pi^2x^2}{12 \log x},\] as a corollary of a result of Alladi and Erd\H{o}s.  This average order is asymptotically the same as both the functions $P_1(n)$ and  $B(n)$. 

Not only does $a(n)$ have the same average order as $P_1(n)$, but in fact we show in Theorem \ref{thm:01result} that on a set of integers with density about $0.623$ the function $P_1(n)$ is a very good approximation to $a(n)$ in the sense that $a(n) = P_1(n) + o(1)$ as $n$ goes to infinity over the integers in this set.  Surprisingly, however, we find that (with the exception of a set of measure 0) on the compliment of this set, where the approximation above is not true, we have that the difference between $a(n)$ and $P_1(n)$ tends to infinity.

As mentioned above, $a(n)$ is bounded between $P_1(n)$ and $B(n)$.  In addition to having the same average order asymptotically, Alladi and Erd\H{o}s show (see Theorem \ref{thm:allerd}) that
\[\sum_{n<x} \Big ( B(n) - P_1(n) \Big ) \sim \sum_{n<x} P_2(n).\]
In other words, the sum of the prime factors of $n$, omitting the largest prime factor is completely dominated by the second largest prime factor.  In light of this result, we consider the behavior of $a(n)-P_1(n)$.  In Theorem \ref{thm:ap1diff} we show that there exist constants $D_\ell$ (whose values are given explicitly in the proof) such that for any fixed integer $N\geq 1$ \[\sum_{n\leq x} a(n) - P_1(n) = \frac{x^{\frac{3}{2}}}{\log^2 x}\sum_{\ell=0}^{N-1} \frac{D_\ell}{\log^\ell x} + O_N\left(\frac{x^{\frac{3}{2}}}{\log^{N+2} x}\right).\]
As Corollary \ref{cor:ap1diff} we get that 
\[\sum_{n<x} \Big ( a(n) - P_1(n) \Big ) \sim \left(1-\frac{\pi}{4}\right)\sum_{n<x} P_2(n),\]
which we can interpret as saying that in an average sense, the value of $a(n)$ is approximately the sum of the largest prime factor of $n$ plus an extra contribution of a factor of $\left(1-\frac{\pi}{4}\right)$ times its second largest prime factor.

Throughout the paper we will use $X$ and $Y$ to denote random variables, $\mathbb{E}[X]$ to denote the expected value of $X$, and $\mathbb{P}$ to denote the probability of an event occurring.

\section{Computing the expected time to reach $0$} \label{sec:recursiveeqn} 
It does not appear to be possible to write down a general, closed formula for $a(n)$.  We are, however, able to give a recursive formula and estimates that are generally quite accurate.  The value of $a(n)$ is rarely an integer, and may never be unless $n$ is a prime power.  We begin with a few cases that are easily dealt with.

In the case when $n=p$ is prime every element of $\Z/p\Z$ besides $0 \pmod{p}$ is a unit.  As a result there is only one way of getting to $0$, namely by randomly choosing the residue $0$ as a multiple when stepping.  This is then a Bernoulli trial with probability of success $\frac{1}{p}$, and thus the expected time to walk to 0 is \[a(p)=p.\]    

The case $n=p^k$, when $n$ is a prime power, is somewhat more enlightening.  As in the prime case, the likelihood, at each step, of randomly selecting a unit (a residue coprime to $p$) is $\frac{p-1}{p}$.  Since the only way to progress toward the absorbing state is to select a non-unit, we expect to wait for $p$ iterations between ``meaningful'' steps that bring us closer to $0$. Unlike the prime case, however, it is unlikely (though possible) that the first step by a non-unit will take us directly to $0$.  It is much more likely that we only pick up one, or a couple of the necessary copies of $p$. In essence, we are taking a walk on the graph in Figure \ref{fig:graphpp}, in which each loop is weighted with probability $\frac{p-1}{p}$, each edge from $p^i$ to $p^j$, $0\leq i <j<k$ is weighted with probability $\frac{p-1}{p^{i-j+1}}$ and each edge from $p^i$ to $p^k$ by $\frac{1}{p^{k-i}}$.

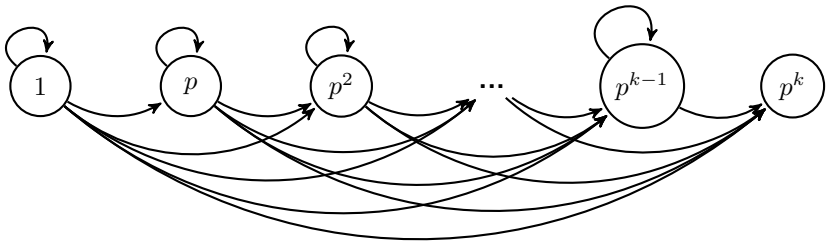
\begin{figure}[h!]
    \centering
    \vspace{-2mm}
    \begin{tikzpicture}[->,>=stealth',shorten >=1pt,auto,node distance=2cm,
  thick,main node/.style={circle,draw,minimum size=0.8cm},   dots/.style={font=\sffamily\large\bfseries}]

  \node[main node] (1) {$1$};
  \node[main node] (2) [right of=1] {$p$};
  \node[main node] (3) [right of=2] {${p^2}$};
  \node[dots] (4) [right of=3] {...};
  \node[main node] (5) [right of=4] {$p^{k{-}1}$};
  \node[main node] (6) [right of=5] {$p^k$};

  \path[every node/.style={font=\sffamily\small}]
    (1) edge [in=80,out=140,loop,looseness=4] node[above] {} (1)
        edge [bend right] node[above] {} (2)
        edge [bend right=40] node[above] {} (3)
        edge [bend right=40] node[above] {} (4)
        edge [bend right=40] node[above,pos=0.56] {} (5)
        edge [bend right=40] node[above,pos=0.62] {} (6)
    (2) edge [in=80,out=140,loop,looseness=4] node [left,pos=0.5]{} (2)
        edge [bend right] node[above] {} (3)
        edge [bend right=40] node[above] {} (4)
        edge [bend right=40] node[above,pos=0.8] {} (5)
        edge [bend right=40] node[above,pos=0.87] {} (6)
    (3) edge [in=80,out=140,loop,looseness=4] node [left,pos=0.5]{} (3)
        edge [bend right] node[above] {} (4)
        edge [bend right=40] node[above] {} (5)
        edge [bend right=40] node[above,pos=0.75] {} (6)
    (4) edge [bend right] node[above] {} (5)
        edge [bend right=40] node[above,pos=0.75] {} (6)
    (5) edge [in=80,out=140,loop,looseness=4] node [left,pos=0.5]{} (5)
        edge [bend right] node[above] {} (6);
\end{tikzpicture}
\vspace{-4mm}
    \caption{The directed graph depicting the random walk when $n=p^k$ is a prime power.  Note that each node in this graph represents all residues that have the same gcd with $n$ rather than individual residues, and that not all edges have the same weight.}
    \label{fig:graphpp}
\end{figure}

\begin{theorem}\label{thm:primepow}
When $n=p^k$ is a power of a prime the expected length of a multiplicative random walk on $\Z/p^k\Z$ is
\begin{equation}a(p^k)=k(p-1)+1. \label{ppower}
\end{equation}
\end{theorem}

\begin{proof}
When $k=1$ the result is the prime case discussed above.  We now proceed by induction, and suppose that equation \ref{ppower} holds for all exponents less than $k$.   By linearity of expectation we can break the expected time up as 
\[a(p^k)= \mathbb{E}[\text{time until first step by non-unit}] + \mathbb{E}[\text{time to step from there to 0}].\]
We know, as above, that the expected number of steps until our first non-unit is $p$, so we need only consider the second.  Once we've taken our first step by a non-unit, say to a residue whose gcd with $p^k$ is $p^i$, we see that we need to pick up an additional $k-i$ copies of $p$.  Doing so amounts to a walk on $\Z/p^{k-i}\Z$, which takes time $a(p^{k-i})$.  

Since the probability that our first non-unit step contains exactly $i$ copies of $p$ is $\frac{p-1}{p^{i}}$ for $i<k$ and $\frac{1}{p^k}$ when $i=k$, we get that \begin{align*}a(p^k) &= p + \sum_{i=1}^{k-1} \frac{p-1}{p^i}a(p^{k-i}) + \frac{a(p^0)}{p^k} \\
&= p + \sum_{i=1}^{k-1} \frac{p-1}{p^i}\left((k-i)(p-1)+1\right).
\end{align*}
Here we have used the induction hypothesis and that $a(1)=0$.  Rewriting this sum, we get that 
\begin{align*}
a(p^k) &= p + \frac{p-1}{p^{k}}\sum_{i=1}^{k-1} p^{k-i}\left((k-i)(p-1)+1\right)\\
&= p + \frac{p-1}{p^{k}}\sum_{i=1}^{k-1} p^{i}\left(i(p-1)+1\right)\\
&= p+ \frac{(p-1)^2}{p^{k}}\sum_{i=1}^{k-1} ip^{i}+ \frac{p-1}{p^{k}}\sum_{i=1}^{k-1} p^{i}\\
&= p+ \frac{(p-1)^2}{p^{k}}\left(\frac{p-p^k}{(1-p)^2} -\frac{(k-1)p^k}{1-p}\right) + \frac{p-1}{p^{k}}\left(\frac{p-p^k}{1-p}\right)\\
&= p+ \frac{p-p^k}{p^{k}} -(k-1)(1-p) - \frac{p-p^k}{p^k}\\
&= k(p-1)+1.\\
\end{align*}
This completes the inductive step and the proof.

\end{proof}

The residues modulo an arbitrary $n$ can be decomposed into the sets \[R_d = \{r \in \Z/n\Z:(r,n)=d\} = d\cdot \left(\Z/\tfrac{n}{d}\Z\right)^\times,\] for each of the divisors $d$ of $n$.  Each such set has size $\varphi(\frac{n}{d})$.  Our random walk on the residues modulo $n$ can also be interpreted as a walk on the classes $R_d$, starting in $R_1=(\Z/n\Z)^\times$. 

The key observation we made at the beginning of this proof to create a recursive formula for $a(p^k)$ works just as well for arbitrary composite $n$, although we aren't able to solve the recurrence.  For arbitrary $n$, the probability of randomly choosing a non-unit is $\frac{n-\varphi(n)}{n}$, where $\varphi(n)$ is the Euler totient function, hence we expect it will take $\frac{n}{n-\varphi(n)}$ steps before our first such choice.  Having done so, the probability that our chosen residue has gcd $d$ with $n$ is $\frac{|R_d|}{n-|R_1|}=\frac{\varphi\left(\frac{n}{d}\right)}{n-\varphi(n)}$ and thus we have the following recursive formula, which allows us to compute $a(n)$ for all $n$.

\begin{theorem} \label{thm:recurform}
For all $n>1$ the expected length of a random walk on $\Z/n\Z$ is 
\begin{align*}
a(n) &= \frac{n}{n-\varphi(n)}+ \sum_{\substack{d|n\\d\neq 1}}\frac{\varphi\left(\frac{n}{d}\right)a\left(\frac{n}{d}\right)}{n-\varphi(n)}\\
&= \frac{1}{n-\varphi(n)}\left(n+ \sum_{\substack{d|n\\d\neq n}}\varphi\left(d\right)a\left(d\right)\right).
\end{align*}
\end{theorem} 

\section{Squarefree integers} \label{sec:squarefree}
In the case that the integer $n$ is squarefree we can solve the recursion in Theorem \ref{thm:recurform}.  It is convenient to use different methods, however, which will be useful later.  In the squarefree case, reaching the residue $0 \pmod{n}$ amounts to having randomly chosen a residue divisible by each prime dividing $n$ at least once in our walk.  For any prime $p$ dividing $n$, we can let $X_p$ be a random variable which counts the number of steps taken before randomly selecting a residue divisible by $p$.  Then all of the random variables $X_p$ are independent, geometric random variables with $\mathbb{E}[X_p] = p$ and 
\[a(n) = \mathbb{E}\left[\max_{p|n}\{X_p\}\right].\]
While the expected value of a maximum of random variables is not, in general, well behaved, it is straightforward to obtain a formula in the case of independent, geometric random variables.  Note that this is possible despite the fact that our variables are not identically distributed.  (Jeske and Blessinger observe \cite{Jeske} that this formula is surprisingly rare in the literature.) 

\begin{theorem} \label{thm:squarefree}
For $n\geq 2$ squarefree
\[a(n) = \sum_{\substack{d|n\\d\neq 1}} (-1)^{\omega(d)+1}\frac{d}{d-\varphi(d)}.\]
\end{theorem}
\begin{proof}
Let $X = \max_{p|n}\{X_p\}$. Then
\begin{align*}
a(n)= \mathbb{E}\left[\max_{p|n}\{X_p\}\right] &= \sum_{i=0}^\infty \mathbb{P}\left[X > i\right] = \sum_{i=0}^\infty \left(1-\mathbb{P}\left[X \leq i\right]\right) \\
&= \sum_{i=0}^\infty \left(1-\prod_{p|n} \mathbb{P}\left[X_p \leq i\right]\right) \\
&= \sum_{i=0}^\infty \left(1-\prod_{p|n} \left(1-\left(\frac{p-1}{p}\right)^i\right) \right)\\
&= \sum_{i=0}^\infty \sum_{\substack{d|n\\d\neq 1}} (-1)^{\omega(d)+1}\frac{\varphi(d)^i}{d^i}\\
&= \sum_{\substack{d|n\\d\neq 1}} \frac{(-1)^{\omega(d)+1}}{1-\frac{\varphi(d)}{d}} = \sum_{\substack{d|n\\d\neq 1}} (-1)^{\omega(d)+1}\frac{d}{d-\varphi(d)}.
\end{align*}
\end{proof}
Note that for $n=p$ this reduces to the equation above, and for $n=pq$, (where $p \neq q$) gives 
\begin{equation} a(pq) = p + q - \frac{pq}{p+q-1} = p + \frac{q(q-1)}{p+q-1}. \label{eq:twoprimes} \end{equation}
This can also be obtained from Theorem \ref{thm:recurform}.

\section{The average order of the expected time} \label{sec:upplowbds}
Because it is necessary to have randomly chosen at least one residue divisible by each prime dividing $n$ and, in particular, we expect the largest prime divisor, $P_1(n)$, to require $P_1(n)$ steps, we know that this function serves as a lower bound for $a(n)$.  

We can obtain an upper bound for $a(n)$ by imagining the situation where we wait for each prime divisor of $n$ separately.  Rather than allowing our random walk to accumulate each prime as it is randomly chosen, we first wait to pick up a copy of $P_1(n)$, then $P_2(n)$, etc, waiting separately for each copy of a prime that divides $n$ multiple times.  Since we expect each prime $p$ dividing $n$ to require $p$ steps, we find that an upper bound for $a(n)$ is the sum of all of the prime divisors of $n$, taken with multiplicity, defined as $B(n)$ in the introduction.  To summarize,
\[P_1(n) \leq a(n) \leq B(n).\]
Now these functions, $P_1(n)$ and $B(n)$, despite arising as relatively crude lower and upper bounds for $a(n)$ aren't so different.  As a matter of fact,  Alladi and Erd\H{o}s studied both of them \cite{AlladiErdos}, and showed that both have the same average order.
\begin{theorem}[Alladi, Erd\H{o}s (1977)] \label{thm:allerd} As $x \to \infty$,
\[\sum_{n<x} P_1(n) \sim \sum_{n<x} B(n) \sim \frac{\pi^2x^2}{12 \log x},\]
and furthermore the difference between these two sums is dominated by the contribution of the second largest prime factor,
\[\left(\sum_{n<x} B(n) - \sum_{n<x} P_1(n) \right) = \sum_{n<x} \Big ( B(n) - P_1(n) \Big ) = \sum_{n<x} P_2(n) +O\left(\frac{x^{4/3}}{\log^3 x}\right)\] 
and that \[\sum_{n<x} P_2(n) \sim C\frac{x^{3/2}}{(\log x)^2}\] 
for some positive constant C. \label{thm:AllErd}
\end{theorem} Ivi\'c \cite{Ivic} notes that the constant $C$ in their result is equal to ${8\zeta\left(\frac{3}{2}\right)}/3$, attributing the result to Balasubramanian.  He also shows that the sum of $P_2(n)$ admits an asymptotic expansion
\begin{theorem}[Ivi\'c, 1990] \label{thm:ivic} There exist constants $A_l$ such that for any integer $N\geq 1$ \[\sum_{n\leq x} P_2(n) = \frac{x^{3/2}}{\log^2 x}\sum_{l=0}^{N-1} \frac{A_l}{\log^l x} +O_N\left(\frac{x^{3/2}}{\log^{N+2} x}\right).\]
\end{theorem}
We will return to Theorem \ref{thm:ivic} in Section \ref{sec:ap1diff}.  For now, we note that we can easily obtain the average order of $a(n)$ from Theorem \ref{thm:allerd}.
\begin{corollary} As $x \to \infty$, \label{cor:avgord}
\[\sum_{n<x} a(n) \sim \frac{\pi^2x^2}{12 \log x}.\]
\end{corollary}  

\section{The typical behaviour of $a(n)$} \label{sec:typical}
In addition to having the same average order, Erd\H{o}s and Pomerance show \cite{ErdosPomerance} that, on a set of asymptotic density one, $P_1(n) \sim B(n)$.  So, in a sense these bounds answer the question of how large $a(n)$ is.  But which function is a better estimate of $a(n)$?  For the majority of integers $n$, $P_1(n)$ turns out to be a remarkably good approximation for $a(n)$, and in fact the difference between these two functions tends to zero on this set. However, for almost all of the remaining integers, a set which still has positive density, the difference $a(n)-P_1(n)$ tends to infinity. 

\begin{theorem} \label{thm:01result}
But for a set of integers, $n$, with asymptotic density $0$, we have that if $P_2(n) < P_1(n)^{1/2}$ then \[a(n) = P_1(n)+o(1)\] as $n \to \infty$, and otherwise, if $P_2(n) > P_1(n)^{1/2}$, then
\[a(n) - P_1(n) \to \infty\] as $n\to\infty$.  The asymptotic density of the
set of $n$ with $P_2(n) < P_1(n)^{1/2}$ has asymptotic density $0.623...$, the
Golomb--Dickman constant.

\end{theorem}
\begin{proof}
The set of integers satisfying the inequality $P_2(n) < P_1(n)^{1/2}$ was studied by Wheeler \cite{wheeler}, \index{Golomb--Dickman constant} who shows that the density exists and is equal to the constant above.  In what follows we disregard those $n$ where any of the following hold:
\begin{enumerate}
\item $P_1(n)<\log n$.
\item $P_1(n)^{1/2}/\log n <P_2(n) <P_1(n)^{1/2} \log  n$. 
\item $n$ is divisible by a composite prime power greater than $P_2(n)$.
\end{enumerate} 
However, each of these conditions occur only on a set of integers with density 0. 

Let $p = P_1(n)$.  By linearity of the expectation, we can write \[a(n) = \mathbb{E}[X_p] + \mathbb{E}[Y],\] where $X_p$ is the random variable counting the number of steps before we randomly choose a residue divisible by $p$, and $Y$ is the time required to pick up any remaining factors after choosing a residue divisible by $p$. Now $\mathbb{E}[X_p] = p$, so it suffices to study $\mathbb{E}[Y]$.  

Let $q_1,q_2,\ldots$ denote the distinct, maximal prime powers dividing $n/p$ and let the random variables $Y_{q_i}$ denote the number of steps in the walk after acquiring a factor of $p$ and before acquiring a factor of $q_i$.  So \[Y = \max_i\{Y_{q_i}\} \leq \sum_i Y_{q_i}.\] 
In what follows we will assume that $q_i$ is prime, as the prime power case will only be smaller.  Because of the memoryless property of the geometric distribution, \index{geometric distribution} we have that 
\begin{align*}
\mathbb{E}[Y_{q_i}] &= \mathbb{E}[X_{q_i}]P(X_p<X_{q_i})  \\
&= q_i\sum_{k=1}^\infty P(X_p =k)P(X_{q_i}>k)  \\
&= q_i\sum_{k=1}^\infty \left(1-\frac{1}{p}\right)^{k-1}\frac{1}{p}\left(1-\frac{1}{q_i}\right)^k  \\
&= \frac{q_i}{p}\left(1-\frac{1}{q_i}\right)\sum_{k=0}^\infty \left(1-\frac{1}{p}\right)^{k}\left(1-\frac{1}{q_i}\right)^k  \\
&= \frac{q_i-1}{p}\left(\frac{1}{1-\left(1-\frac{1}{p}\right)\left(1-\frac{1}{q_i}\right)}\right) \\
&= \frac{q_i-1}{1+\frac{p}{q_i} -\frac{1}{q_i}} < \frac{q_i^2}{p}.
\end{align*}
Now, in the first case, where $P_2(n) < P_1(n)^{1/2}$, we have that $q_i \leq P_2(n) \leq \sqrt{p}/\log n$ (since we disregarded the set of integers whose second largest prime factor was smaller than $\sqrt{p}$ by less than a factor of $\log n$) and so the  expression above is less than $(\log n)^{-2}$. Since $n$ has at most $\log n$ distinct prime factors, we have that
\begin{align*}
\mathbb{E}[Y] \ &\leq \ \sum_{i}\mathbb{E}[Y_{q_i}] \ < \ \frac{\log n}{(\log n)^2} \ < \ \frac{1}{\log n} \ = \ o(1)
\end{align*}
as $n \to \infty$. 

In the second case where $P_2(n)^2 > P_1(n)$, we can assume that $q=P_2(n) > p^{1/2}\log n$, as we also disregarded above the set of integers where this was not the case, and so 
\begin{align*}
\mathbb{E}[Y_q] &=  \frac{q-1}{1+\frac{p}{q} -\frac{1}{q}} > \frac{q-1}{1+\frac{p}{q}} \\
&\gg \frac{q}{1+\frac{p}{q}} \gg \frac{q}{p/q} = \frac{q^2}{p} > \log^2 n.
\end{align*}
Since this goes to infinity as $n$ does, we have that $a(n)-P_1(n) \geq \mathbb{E}[Y_q]$ does as well.
\end{proof}

The same method of proof used here can also be extended to show, for example, that (but for a set of density zero) on the integers where $P_3(n)^3< P_1(n)\cdot P_2(n)$ we have $a(n) = P_1(n)+\frac{P_2(n)^2}{P_1(n)+P_2(n)} + o(1)$ as $n \to \infty$.  Note that the set of integers on which this condition holds both contains and has density strictly greater than the set where $P_2(n)^2 < P_1(n)$.  

It similarly follows that for each fixed $k$ there exists a set of integers on which the value of $a(n)$ is approximated, up to an error term that is $o(1)$, by a function involving only by the largest $k$ prime factors of $n$, where the density of these sets approaches $1$ as $k$ goes to infinity. However the expression for $a(n)$ in terms of these prime factors becomes increasingly complicated.

\section{The average time spent waiting after the largest prime factor} \label{sec:ap1diff}
We know that $P_1(n)\leq a(n) \leq B(n)$, with equality holding in each case if and only if $n$ is prime. In light of Theorems \ref{thm:allerd} and \ref{thm:ivic} it makes sense to similarly consider the difference $a(n)-P_1(n)$, which also corresponds, by linearity of expectation, to the expected number of steps spent waiting after the random walk has picked up a factor of the largest prime divisor of $n$.  
\begin{theorem} \label{thm:ap1diff}
There exist constants $D_\ell$ such that, for any fixed integer $N \geq 1$,
\[\sum_{n\leq x}\big( a(n) - P_1(n) \big)= \frac{x^{\frac{3}{2}}}{\log^2 x}\sum_{\ell=0}^{N-1} \frac{D_\ell}{\log^\ell x} + O_N\left(\frac{x^{\frac{3}{2}}}{\log^{N+2} x}\right).\]
In particular, $D_0 = \frac{8\zeta\left(\frac{3}{2}\right)}{3}\left(1-\frac{\pi}{4}\right)$, and an expression that can be used to compute $D_\ell$ for any value of $\ell$ is given in the proof \eqref{eq:constants} below. \label{thm:amp}
\end{theorem}
Comparing this to Theorem \ref{thm:AllErd} and the constant found by Balasubramanian (which can also be obtained by methods similar to the proof of Theorem \ref{thm:amp}) we can get a result about how $a(n)$ depends, in an average sense, on its largest two prime factors.
\begin{corollary} \label{cor:ap1diff}
The difference between $a(n)$ and $P_1(n)$ satisfies \[\sum_{n\leq x} \big( a(n) - P_1(n) \big) \sim \left(1-\frac{\pi}{4}\right)\sum_{n\leq x} P_2(n).\]
\end{corollary}

\begin{proof}[Proof of Theorem \ref{thm:amp}]
Note first that if $n=p$ is prime then $a(p)-P_1(p) = 0$, so the primes contribute nothing to this sum.  Using $p$ and $q$ to denote prime numbers we can then write \begin{align*}
    \sum_{n\leq x} \big(a(n) - P_1(n)\big) &= \sum_{p<x}\sum_{q\leq p} \sum_{\substack{m\leq x/pq\\P_1(m) \leq q}} \left(a(pqm)-p\right) \\
    &= \sum_{p<x}\sum_{q\leq p} \sum_{\substack{m\leq x/pq\\P_1(m) \leq q}} \big(a(pq)-p +O(a(m))\big)
\end{align*}
since $a(pq) \leq a(pqm) \leq a(pq) + a(m)$.  We can bound the error term above as \begin{align*}
\sum_{p<x}\sum_{q\leq p} \sum_{\substack{m\leq x/pq\\P_1(m) \leq q}} a(m) &\leq \sum_{p<x}\sum_{q\leq p} \sum_{\substack{m\leq x/pq\\P_1(m) \leq q}} B(m)\\
&= \sideset{}{'}\sum_{n\leq x}  \big( B(n) - P_1(n) - P_2(n) \big)  =O\left(\frac{x^{4/3}}{\log^3 x}\right)
\end{align*}
using Theorem \ref{thm:allerd}. Here $\sum^{'}$ denotes that the sum is over composite integers.  Dividing the main sum into two sums based on whether the largest prime divisor divides the integer twice we get that \begin{align*}
    \sum_{n\leq x} \big(a(n) - P_1(n)\big) &=  \sum_{p\leq \sqrt{x}} \sum_{\substack{m\leq x/p^2\\P_1(m) \leq p}} a(p^2)-p +  \sum_{p<x}\sum_{q < p} \sum_{\substack{m\leq x/pq\\P_1(m) \leq q}} a(pq)-p+O\left(x^{4/3}\right) \\
    &= \sum_{p\leq \sqrt{x}} \sum_{\substack{m\leq x/p^2\\P_1(m) \leq p}} \big( p - 1 \big) +  \sum_{p<x}\sum_{q < p} \sum_{\substack{m\leq x/pq\\P_1(m) \leq q}} \frac{q(q-1)}{p+q-1} +O\left(x^{4/3}\right). \\ 
\end{align*} 

The first sum above is 
\begin{align*}
    \sum_{p\leq \sqrt{x}} \sum_{\substack{m\leq x/p^2\\P_1(m) \leq p}} \big( p - 1 \big) \leq \sum_{p\leq \sqrt{x}} p \left\lfloor\frac{x}{p^2}\right\rfloor \leq \sum_{p\leq \sqrt{x}} \frac{x}{p} = O\left(x \log \log x\right),
\end{align*}
so it can be absorbed into the error term.  This leaves the second sum, which we simplify and reorder as \begin{align*}
    \sum_{p<x}\sum_{q < p} \sum_{\substack{m\leq x/pq\\P_1(m) \leq q}} \frac{q(q-1)}{p+q-1} &= \sum_{p<x}\sum_{q < p} \sum_{\substack{m\leq x/pq\\P_1(m) \leq q}} \left(\frac{q^2}{p+q} - \frac{q}{p+q-1} +\frac{q^2}{(p+q)(p+q-1)}\right)\\
    &= \sum_{p<x}\sum_{q < p} \sum_{\substack{m\leq x/pq\\P_1(m) \leq q}} \frac{q^2}{p+q} +O(x)\\
    &=\sum_{m\leq x/6} \sum_{\substack{p<x/m\\ p>P_1(m)}} \sum_{\substack{q <p \\ q \leq x/mp \\ q\geq P_1(m)}} \frac{q^2}{p+q} + O(x) \\
    &=\sum_{m\leq x/6} \sum_{\substack{p\leq \sqrt{\frac{x}{m}}\\ p>P_1(m)}} \sum_{\substack{q <p\\ q\geq P_1(m)}} \frac{q^2}{p+q} + \sum_{m\leq x/6} \sum_{\substack{\sqrt{\frac{x}{m}} <p<\frac{x}{m} \\ p>P_1(m)}} \sum_{\substack{q \leq \frac{x}{mp} \\ q\geq P_1(m)}} \frac{q^2}{p+q} + O(x) \\
    &=\hspace{5mm} S_1 \hspace{5mm} + \hspace{5mm} S_2 \hspace{5mm} + O(x) \numberthis \label{eq:s1s2}.
\end{align*}

We evaluate the two sums $S_1$ and $S_2$ separately. Taking first $S_1$, we find that the sum is dominated by small values of $m$. The contribution to $S_1$ by those $m \geq x^\epsilon$ for any fixed small $0<\epsilon<\frac{1}{4}$ is
\begin{align*}
     \sum_{x^\epsilon \leq m \leq x/6} \sum_{\substack{p\leq \sqrt{\frac{x}{m}}\\ p>P_1(m)}} \sum_{\substack{q <p\\ q\geq P_1(m)}} \frac{q^2}{p+q} &\leq \sum_{x^\epsilon < m \leq x/6} \sum_{\substack{p\leq \sqrt{\frac{x}{m}}}} \sum_{\substack{q <p}} q  \\
    & \ll \sum_{x^\epsilon \leq m \leq x/6} \sum_{\substack{p\leq \sqrt{\frac{x}{m}}}} \frac{p^2}{\log p} =  \ll \sum_{x^\epsilon \leq m \leq x/6} \left(\frac{x}{m}\right)^{3/2} =O\left(x^{\frac{3-\epsilon}{2}}\right).
\end{align*}
and likewise the contribution from either $p$ or $q \leq x^\epsilon$ is insignificant, so we can write 

\begin{equation}
S_1 = \sum_{m<x^\epsilon} \sum_{\substack{x^\epsilon < p\leq \sqrt{\frac{x}{m}}}} \sum_{\substack{x^\epsilon < q <p}} \frac{q^2}{p+q} + O\left(x^{\frac{3-\epsilon}{2}}\right). \label{eq:cleans1} 
\end{equation}

We treat the inner summation using partial summation 
\begin{align*}
    \sum_{\substack{x^\epsilon < q <p}} \frac{q^2}{p+q}  &= \int_{x^\epsilon}^p \frac{s^2}{(p+s)\log s} ds + O\left(p^2e^{-\sqrt{\log p}}\right).
\end{align*}
We tackle the integral by expanding the fraction as a geometric series
\[\frac{s^2}{(p+s)\log s} = \frac{s^2}{p\log s}\sum_{i=0}^\infty \left(-\frac{s}{p}\right)^i\] (which converges for all $s<p$) and integrate term by term,\

\begin{align*}
    \int_{x^\epsilon}^p \frac{s^2}{(p+s)\log s} ds &= \sum_{i=0}^\infty \left( \frac{(-1)^i}{p^{i+1}} \int_{x^\epsilon}^p \frac{s^{i+2}}{\log s} ds\right)\\
    &=  \sum_{i=0}^\infty \left(\frac{(-1)^i}{p^{i+1}} \left(\li\left(p^{i+3}\right)-\li\left(x^{\epsilon(i+3)}\right)\right) \right)\\
    &=  \sum_{i=0}^\infty \left(\frac{(-1)^i}{p^{i+1}} \left(\frac{p^{i+3}}{\log p^{i+3}}\left(\sum_{k=0}^{N-1}\frac{k!}{\log^k p^{i+3}}+O_N\left(\frac{1}{\log^{N} p^{i+3}}\right)\right)+O\left(\frac{x^{\epsilon(i+3)}}{\epsilon(i+3)\log x}\right)\right)\right) \\
    &=  \frac{p^2}{\log p}\sum_{i=0}^\infty \sum_{k=0}^{N-1} \left(\frac{(-1)^i k!}{(i+3)^{k+1}\log^k p}\right)+O_N\left(\frac{p^2}{\log^{N+1} p}\right).
\end{align*}
Here we have used the asymptotic expansion for $\li(x)$, where $N\geq 0$ is a fixed integer. Plugging this into \eqref{eq:cleans1} and applying partial summation to the variable $p$, we get 
\begin{align*}
    S_1 &=  \sum_{m<x^\epsilon}\left( \sum_{x^{\epsilon}<p<\sqrt{\frac{x}{m}}}\left(\frac{p^2}{\log p}\sum_{i=0}^\infty \sum_{k=0}^{N-1}\left(\frac{(-1)^i k!}{(i+3)^{k+1}\log^k p}\right)+O_N\left(\frac{p^2}{\log^{N+1} p}\right)\right) \right)\\
    &= \sum_{m<x^\epsilon} \left( \int_{x^\epsilon}^{\sqrt{\frac{x}{m}}} \frac{t^2}{\log^2 t}\sum_{i=0}^\infty \sum_{k=0}^{N-1}\left(\frac{(-1)^i k!}{(i+3)^{k+1}\log^k t}\right)dt+O_N\left(\frac{(x/m)^{3/2}}{\log^{N+2} \frac{x}{m}}\right)\right). \numberthis \label{eq:qintegral1} \\
\end{align*}

Integrating by parts one finds that \begin{align*}
\int \frac{t^2}{\log^{k+2} t}dt &= \frac{3^{k+1}}{(k+1)!}\li(t^3)-t^3\sum_{h=0}^{k}\frac{3^h(k-h)!}{(k+1)!\log^{h+1} t} +C \\
&= t^3\sum_{h=k+1}^{N}\frac{3^{k-h}h!}{(k+1)!\log^{h+1} t}+O_N\left(\frac{t^3}{\log^{N+2} t}\right)
\end{align*}
since the sum in the first line above simply subtracts off the first $k+1$ terms in the asymptotic expansion of $ \frac{3^{k+1}}{(k+1)!}\li(t^3)$.
Using this to evaluate the integral in \eqref{eq:qintegral1} we obtain
\begin{align*}
    \left. \sum_{i=0}^\infty \sum_{k=0}^{N-1}\left(\frac{(-1)^i k!}{(i+3)^{k+1}}\left(t^3\sum_{h=k+1}^{N}\frac{3^{k-h}h!}{(k+1)!\log^{h+1} t}+O_N\left(\frac{t^3}{\log^{N+2} t}\right)\right)\right)\right|_{t=x^\epsilon}^{t=\sqrt{\frac{x}{m}}}
\end{align*}

The lower limit of integration can easily be absorbed into the error term, so we need only concern the upper limit.  Evaluating this expression there, it becomes 

\begin{align*}
    &\sum_{i=0}^\infty \sum_{k=0}^{N-1}\left(\frac{(-1)^i k!}{(i+3)^{k+1}}\left(\frac{x}{m}\right)^{3/2}\sum_{h=k+1}^{N}\frac{3^{k-h}h!2^{h+1}}{(k+1)!\log^{h+1} \frac{x}{m}}\right)\\
    &=2\left(\frac{x}{m}\right)^{3/2}\sum_{i=0}^\infty (-1)^i\sum_{h=1}^{N}\frac{2^h h!}{\log^{h+1} \frac{x}{m}} \sum_{k=0}^{h-1}\frac{3^{k-h}}{(k+1)(i+3)^{k+1}} \\
    &=2\left(\frac{x}{m}\right)^{3/2} \sum_{h=0}^{N-1}\frac{2^{h+1} h!}{\log^{h+2} \frac{x}{m}} \sum_{k=0}^{h}\frac{3^{k-h-1}}{k+1}\sum_{i=3}^\infty \frac{(-1)^{i+1}}{i^{k+1}}. \numberthis \label{eq:s1pint}
\end{align*}

We therefore get the estimate \[S_1 = 2\sum_{m<x^\epsilon}\left(\left(\frac{x}{m}\right)^{3/2} \left(\sum_{h=0}^{N-1}\frac{2^{h+1} h!}{\log^{h+2} \frac{x}{m}} \sum_{k=0}^{h}\frac{3^{k-h-1}}{k+1}\sum_{i=3}^\infty \frac{(-1)^{i+1}}{i^{k+1}} +O_N\left(\frac{1}{\log^{N+2} \frac{x}{m}}\right)\right)\right).\]

We now treat $S_2$.  As was the case with $S_1$, we can ignore the contribution from those $m>x^\epsilon$, $q\geq x^\epsilon$ or $p>x^{1-\epsilon}$, so 
\begin{equation}
S_2 = \sum_{m < x^\epsilon} \sum_{\substack{\sqrt{\frac{x}{m}}< p < x^{1-\epsilon}}} \sum_{\substack{x^\epsilon < q <\frac{x}{mp}}} \frac{q^2}{p+q} + O\left(x^{\frac{3-\epsilon}{2}}\right). \label{eq:cleans2} 
\end{equation}
We also treat the innermost sum similarly to $S_1$, resulting in an integration similar to \eqref{eq:qintegral1} but with different limits of integration.

\begin{align*}
    \sum_{\substack{x^\epsilon < q <\frac{x}{mp}}} \frac{q^2}{p+q} &=\int_{x^\epsilon}^{\frac{x}{mp}} \frac{s^2}{(p+s)\log s} ds = \sum_{i=0}^\infty \left( \frac{(-1)^i}{p^{i+1}} \int_{x^\epsilon}^{\frac{x}{mp}} \frac{s^{i+2}}{\log s} ds\right)\\
    &=  \sum_{i=0}^\infty \left(\frac{(-1)^i}{p^{i+1}} \left(\li\left(\left({\frac{x}{mp}}\right)^{i+3}\right)-\li\left(x^{\epsilon(i+3)}\right)\right) \right)\\
    &=  \sum_{i=0}^\infty \left(\frac{(-1)^i}{p^{i+1}} \left(\li\left(\left({\frac{x}{mp}}\right)^{i+3}\right)\right)\right) +O\left(\frac{x^{3\epsilon}}{p} \right).
\end{align*}
Expanding the main term above using the asymptotic expansion for the logarithmic integral we get
\begin{align*}
    \sum_{i=0}^\infty &\left(\frac{(-1)^i}{p^{i+1}}  \left(\sum_{k=0}^{N-1}\frac{x^{i+3}k!}{(mp)^{i+3}\log\left( \left({\frac{x}{mp}}\right)^{i+3}\right)^{k+1}}\right)\right)+O_N\left(\sum_{i=0}^\infty\frac{(-1)^ix^{i+3}}{p^{i+1}(mp)^{i+3}\log^{N+1} \left({\frac{x}{mp}}\right)^{i+3}}\right) \\
    &=  p^2\sum_{i=0}^\infty \left(\frac{x}{mp^2}\right)^{i+3} \sum_{k=0}^{N-1}\left(\frac{(-1)^i k!}{(i+3)^{k+1} \log^{k+1} \frac{x}{m} \left(1 - \frac{\log p}{\log \frac{x}{m}}\right)^{k+1} }\right)+O_N\left(\frac{x^3}{m^3p^4\log^{N+1} \frac{x}{mp}}\right) \\
    &=  p^2\sum_{i=3}^\infty \left(\frac{x}{mp^2}\right)^{i} \sum_{k=0}^{N-1}\left(\frac{(-1)^{i+1} k!}{i^{k+1} \log^{k+1} \frac{x}{m} }\sum_{j=0}^\infty\left( \binom{k{+}j}{j} \left(\frac{\log p}{\log \frac{x}{m}}\right)^j\right)\right){+}O_N\left(\frac{(x/m)^3}{p^4\log^{N+1} \frac{x}{mp}}\right). \numberthis \label{eq:s2qintegral} 
\end{align*}

Here we have used the series expansion $\left(\frac{1}{1-x}\right)^{k+1} = \displaystyle{\sum_{j=0}^\infty \left(\binom{k+j}{j} x^j\right)}$.  Inserting this into \eqref{eq:cleans2}, and applying partial summation to p we get that

\begin{align*}
    S_2 &= \sum_{m<x^\epsilon} \left(\int_{\sqrt{\frac{x}{m}}}^{x^{1-\epsilon}} \frac{t^2}{\log t}\sum_{i=3}^\infty \left(\frac{x}{mt^2}\right)^{i} \sum_{k=0}^{N-1}\frac{(-1)^{i+1} k!}{i^{k+1} \log^{k+1} \frac{x}{m} }\sum_{j=0}^\infty\left( \binom{k{+}j}{j} \left(\frac{\log t}{\log \frac{x}{m}}\right)^j\right)dt +O_N\left(\frac{(x/m)^{3/2}}{\log^{N+2} \frac{x}{m}}\right)\right)\\
    &= \sum_{m<x^\epsilon} \left( \sum_{i=3}^\infty \left((-1)^{i+1}\left(\frac{x}{m}\right)^i \sum_{k=0}^{N-1}\left(\frac{(-1)^{i+1} k!}{i^{k+1} \log^{k+1} \frac{x}{m} }\int_{\sqrt{\frac{x}{m}}}^{x^{1-\epsilon}}  \frac{1}{t^{2i-2}\log t} \sum_{j=0}^\infty\left( \binom{k{+}j}{j} \left(\frac{\log t}{\log \frac{x}{m}}\right)^j\right)dt\right)\right)\right.\\
    &\hspace{11cm} \left. \vphantom{ \sum_{j=0}^\infty }+O_N\left(\frac{(x/m)^{3/2}}{\log^{N+2} \frac{x}{m}}\right)\right). \numberthis \label{eq:s2pint}
\end{align*}

Using the fact that
\begin{equation}
\int \frac{\log^{j-1} p}{t^{2i-2}} dt = 
\begin{cases} \li\left(t^{-2i+3}\right) + C &\mbox{if } j=0 \\ 
-\displaystyle{\frac{\log^{j-1} t}{(2i-3)t^{2i-3}} \sum_{l=0}^{j-1} \frac{j!}{(j-l)!(2i-3)^l \log^l t}} + C & \mbox{if } j \geq 1 \end{cases}
\end{equation} 
the integral in the expression above becomes
\begin{align*}
&\left. \left(\li\left(\frac{1}{t^{2i-3}}\right)-\frac{1}{t^{2i-3}}\sum_{j=1}^\infty \left(\binom{k+j}{j} \left(\frac{1}{\log \frac{x}{m}}\right)^j\sum_{l=0}^{j-1} \frac{j!\log^{j-1} t}{(j-l)!(2i-3)^{l+1} \log^l t} 
 \right)\right)\right|_{t=\sqrt{\frac{x}{m}}}^{t=x^{1-\epsilon}}.
\end{align*}
When evaluated at $t=x^{1-\epsilon}$, the above expression can be absorbed into the existing error term, so we need only evaluate this expression when $t=\sqrt{\frac{x}{m}}$, in which case it becomes
\begin{align*}
&-\left(\li\left(\left(\frac{x}{m}\right)^{-\frac{2i-3}{2}}\right){-}\left(\frac{x}{m}\right)^{-\frac{2i-3}{2}}\sum_{j=1}^\infty \left(\binom{k{+}j}{j} \frac{1}{2^{j-1} \log \frac{x}{m}}\sum_{l=0}^{j-1} \frac{j!2^l }{(j{-}l)!(2i{-}3)^{l+1} \log^l \frac{x}{m}} \right) \right)\\
&=-\left(\li\left(\left(\frac{x}{m}\right)^{-\frac{2i{-}3}{2}}\right){-}\left(\frac{x}{m}\right)^{-\frac{2i-3}{2}}\sum_{l=0}^\infty \frac{2^{l+1} }{(2i{-}3)^{l+1} \log^{l+1} \frac{x}{m}}\sum_{j=l+1}^{\infty} \left(\binom{k{+}j}{j}  \frac{j! }{2^j(j{-}l)!} \right)\right). \numberthis \label{eq:s2expanded}
\end{align*}
The sum over $j$ can be evaluated as $\frac{\left(2^{k+1}-2^{-l}\right) ({l+k})!}{k!}$, and the sum over $l$ can be truncated at $N-k-1$, with the remainder of the sum being absorbed into the existing error term, $O_N\left(\frac{(x/m)^{3/2}}{\log^{N+2} \frac{x}{m}}\right)$.  We can also again use the asymptotic expansion of the logarithmic integral, again truncating at $N-k-1$, which turns the expression in \eqref{eq:s2expanded} into 
\begin{align*}
\left(\frac{x}{m}\right)^{-\frac{2i-3}{2}}\sum_{l=0}^{N-k-1}\left(\frac{2}{(2i-3)^{l+1} \log^{l+1} \frac{x}{m}}\left(l!(-1)^{l}2^l+\frac{(2^{l+k+1}-1)(l+k)!}{k!}\right)\right).
\end{align*}
Inserting this back into \eqref{eq:s2pint} in place of the integral, and reindexing so that the summation runs over the exponent on the logarithm we get that the main term is
\begin{align*}
\sum_{m<x^\epsilon}\left(\frac{x}{m}\right)^{\frac{3}{2}}\sum_{i=3}^\infty \sum_{h=0}^{N-1}\left( \frac{(-1)^{i+1}}{\log^{h+2} \frac{x}{m}}\sum_{k=0}^h\left(\frac{2}{i^{k+1}(2i-3)^{h-k+1}}\left(k!(h-k)!(-1)^{h-k}2^{h-k}+(2^{h+1} -1)h!\right)\right)\right).
\end{align*}

Combining that with the corresponding estimate, \eqref{eq:s1pint}, we have that the entire main term of the sum that we are interested in, \eqref{eq:s1s2} is equal to 
\begin{align*}
 \sum_{m\leq x/6}\left(\frac{x}{m}\right)^{\frac{3}{2}} \sum_{h=0}^{N-1}\frac{1}{\log^{h+2} \frac{x}{m}} \sum_{k=0}^{h}\sum_{i=3}^\infty\frac{2(-1)^{i+1}}{i^{k+1}} \left(\frac{2^{h+1}3^{k-h-1} h!}{(k+1)} +  \frac{\left(k!(h-k)!(-2)^{h-k}+\left(2^{h+1}-1\right)h!\right)}{(2i-3)^{h+1-k}}\right)
\end{align*}
with an error term of $O_N\left(\frac{(x/m)^{3/2}}{\log^{N+2} \frac{x}{m}}\right)$.  Note that the quantity given by the double summation over $k$ and $i$ is a constant depending only on $h$, which we call $C_h$. Thus 
\begin{align*}
    \sum_{n\leq x} \big(a(n) - P_1(n)\big) &= x^{\frac{3}{2}} \sum_{m\leq x^\epsilon}\left( \frac{1}{m^{\frac{3}{2}}}\sum_{h=0}^{N-1}\frac{C_h}{\log^{h+2} \frac{x}{m}} +O_N\left(\frac{1}{m^{\frac{3}{2}}\log^{N+2} \frac{x}{m}}\right)\right)\\
    &= x^{\frac{3}{2}} \sum_{h=0}^{N-1} C_h \sum_{m\leq x^\epsilon} \frac{1}{m^{\frac{3}{2}}}\frac{1}{\log^{h+2} \frac{x}{m}} +O_N\left(\frac{x^{\frac{3}{2}}}{\log^{N+2} x}\right).\\
\end{align*}
Since 
\begin{align*}
    \sum_{m\leq x^\epsilon} \frac{1}{m^{\frac{3}{2}}\log^{h+2} \frac{x}{m}} &= \sum_{m\leq x^\epsilon} \frac{1}{m^{\frac{3}{2}}\log^{h+2} x}\left(\frac{1}{1-\frac{\log m}{\log x}}\right)^{h+2}  \\
    &= \sum_{m\leq x^\epsilon} \frac{1}{m^{\frac{3}{2}}\log^{h+2} x}\sum_{j=0}^\infty\binom{h+1+j}{j}\left(\frac{\log m}{\log x}\right)^j\\
    &=\sum_{j=0}^\infty \frac{\binom{h+j+1}{j}}{\log^{h+j+2} x}\sum_{m\leq x^\epsilon} \frac{\log^j m}{m^{\frac{3}{2}}} =\sum_{j=0}^{N-h-1} \frac{\binom{h+j+1}{j}}{\log^{h+j+2} x}(-1)^j\zeta^{(j)} \left(\tfrac{3}{2}\right) +O_N\left(\frac{1}{\log^{N+2} x}\right),
\end{align*}
where $\zeta^{(j)}\left(\frac{3}{2}\right)$ denotes the $j$th derivative of the zeta function at $\frac{3}{2}$, we can conclude that 
\begin{align*}
     \sum_{n\leq x} \big(a(n) - P_1(n)\big) &= x^{\frac{3}{2}} \sum_{h=0}^{N-1} C_h \sum_{j=0}^{N-h-1} \frac{\binom{h+j+1}{j}}{\log^{h+j+2} x}(-1)^j\zeta^{(j)} \left(\tfrac{3}{2}\right) +O_N\left(\frac{x^{\frac{3}{2}}}{\log^{N+2} x}\right)\\
     &= x^{\frac{3}{2}} \sum_{\ell=0}^{N-1}\frac{1}{\log^{\ell+2} x}\sum_{h=0}^\ell \binom{\ell+1}{\ell-h} C_h (-1)^j\zeta^{(\ell-h)} \left(\tfrac{3}{2}\right) +O_N\left(\frac{x^{\frac{3}{2}}}{\log^{N+2} x}\right)\\
     &= x^{\frac{3}{2}} \sum_{\ell=0}^{N-1}\frac{D_\ell}{\log^{\ell+2} x} +O_N\left(\frac{x^{\frac{3}{2}}}{\log^{N+2} x}\right)
\end{align*}
where 
\begin{equation}
D_\ell = 2\sum_{h=0}^\ell \binom{\ell{+}1}{\ell{-}h}\zeta^{(\ell-h)} \left(\tfrac{3}{2}\right)  \sum_{k=0}^{h}\sum_{i=3}^\infty\frac{(-1)^{\ell-h+i+1}}{i^{k+1}} \left(\frac{2^{h+1} h!}{3^{h-k+1}(k+1)} +  \frac{\left(k!(h{-}k)!(-2)^{h-k}{+}\left(2^{h+1}{-}1\right)h!\right)}{(2i-3)^{h+1-k}}\right). \label{eq:constants}
\end{equation}
In the case $\ell=0$ this reduces to
\begin{align*}
    D_0 &= 2\zeta\left(\tfrac{3}{2}\right)\sum_{i=3}^\infty \frac{(-1)^{i+1}}{i}\left(\frac{2}{3}+\frac{2}{2i-3}\right)\\
    &= 2\zeta\left(\tfrac{3}{2}\right)\sum_{i=3}^\infty (-1)^{i+1}\left(\frac{2}{3i}+\left(-\frac{2}{3i}+\frac{4}{3(2i-3)}\right)\right)\\
    &= \frac{8\zeta\left(\tfrac{3}{2}\right)}{3}\sum_{i=3}^\infty \frac{(-1)^{i+1}}{2i-3} = \frac{8\zeta\left(\tfrac{3}{2}\right)}{3}\left(1-\frac{\pi}{4}\right).
\end{align*}
\end{proof}
\section{Open questions and acknowledgements}
There are many remaining questions and directions that this research could be further pursued.  A few examples of potentially interesting questions are given below.
\begin{question}
Is $a(n)$ ever an integer when $n$ is not a prime power?
\end{question}
Computations so far have failed to find any examples of such an integer.
\begin{question}
How many distinct residues modulo $n$ is this random walk expected to visit?
\end{question}
\begin{question}
What can be said about the variance of the time required for this random walk?
\end{question}
\begin{question}
How do these results change if the residues chosen randomly to multiply the current state of the walk are chosen without replacement?
\end{question}

A portion of this work appeared in my Ph.D. thesis \cite{mcnewthesis}, written under the direction of Carl Pomerance, whose helpful suggestions were invaluable in the development of this paper. I would
also like to thank Sarah Wolff, whose Graduate Student Seminar talk on random walks on monoids inspired questions which led to the results in this paper, and Peter Winkler for helpful discussions.
\bibliographystyle{amsplain}
\bibliography{bibliography}

\end{document}